\documentclass{amsart}
\usepackage[utf8]{inputenc}
\usepackage{amsmath, amssymb, amsthm}
\usepackage[english]{babel}
\usepackage{xcolor}
\usepackage{pgf,tikz-cd}
\usepackage{algorithmic}
\usepackage[ruled,vlined]{algorithm2e}

\def\K{{ \mathbb{K}}}
\def\R{\mathbb{R}}

\newcommand{\minor}[1]{\textcolor{blue}{#1}}

\newcommand{\p}{\mathcal{P}}

\newcounter{thm}
\newcounter{ex}
\newcounter{re}
\newtheorem{Theorem}[thm]{Theorem}

\newtheorem{Proposition}[thm]{Proposition}
\newtheorem{example}[thm]{Example}
\newtheorem{remark}[thm]{Remark}
\newtheorem{Definition}[thm]{Definition}

\title{Linearizations and optimization problems in diffeological spaces}
\author{Jean-Pierre Magnot}
\date{}
\address{{LAREMA, Universit\'e d’Angers, 2 Bd Lavoisier, 
49045 Angers cedex 1, France;  Lyc\'ee Jeanne d'Arc, 40 avenue de Grande Bretagne, 63000 Clermont-Ferrand, 
France}; Lepage Research Institute, 17 novembra 1, 081 16 Presov, Slovakia}
\email{\small magnot@math.cnrs.fr; jean-pierr.magnot@ac-clermont.fr}

\begin{document}

\begin{abstract}
By generalizing the notion of linearization, a concept originally arising from microlocal analysis and symbolic calculus, to diffeological spaces, we make a first proposal setting for optimization problems in this category. 

We show how linearizations allow the construction of smooth paths and variational flows without requiring canonical charts or gradients. With these constructions, 
we introduce a general optimization algorithm adapted to diffeological spaces under weakened assumptions. The method applies to spaces of mappings with low regularity.

Our results show that weak convergence toward minima or critical values can still be achieved under diffeological conditions. The approach extends classical variational methods into a flexible, non-linear infinite-dimensional framework. Preliminary steps to the search for fixed points of diffeological mappings are discussed.
\end{abstract}
\maketitle


\section*{Introduction}

The theory of diffeological spaces provides a natural framework for analyzing smooth structures beyond the realm of finite-dimensional manifolds, particularly in contexts where classical differential geometry is insufficient. This flexibility suggests diffeologies to be an essential tool in the study of mapping spaces, infinite-dimensional geometry, and topological vector spaces with weak topologies.

One of the central challenges in applying differential techniques in such settings lies in the absence of canonical local charts or metrics, which are typically used in optimization and variational analysis on manifolds. \minor{See also \cite{Boumal2023,Absil2008}.} 
In this article, we explore and generalize the notion of linearization, originally introduced by Bokobza-Haggiag in the study of pseudo-differential operators on manifolds, and extended later by Widom through symbolic calculus in microlocal analysis. \minor{\cite{Wid}} \minor{\cite{BK}} We reinterpret these constructions along the lines of diffeological geometry, and show how they can be leveraged to define optimization algorithms on diffeological spaces. In particular, we introduce a class of minimization methods in the spirit of well-known steepest descent methods, that can be intuitively understood as generalizations of gradient-type methods that do not require explicit knowledge of gradients, but only of a functional and its evaluation along smooth paths.

We then apply this framework to the setting of mapping spaces with low regularity, such as Sobolev spaces \( W^{s,p}(M,N) \), where \( M \) and \( N \) are compact manifolds and \( s \in \mathbb{R} \), \( 1 < p < \infty \). We analyze how linearizations of the target manifold \( N \) naturally induce linearizations on the corresponding mapping spaces, and establish convergence results for the associated variational algorithm under weak regularity assumptions. This allows us to treat optimization problems in topologies that lack Banach manifold structures, extending classical results into the diffeological domain. A discussion on the construction of invariant spaces, or fixed points, of smooth functions ends the study.

\section{Preliminaries on diffeologies}
\subsection{Diffeologies and Fr\"olicher spaces}

In this section, we follow the presentations from \cite{Ma2013, MR2016}, and complete the classical exposition \cite{Igdiff} along the lines of the recent survey \cite{GMW2023}. 

\begin{Definition}
	Let $X$ be a set.
	
	\begin{itemize}
		\item A \textbf{$p$-parametrization} of dimension $p$ on $X$ is a map from an open subset $O \subset \R^p$ to $X$.
		
		\item A \textbf{diffeology} on $X$ is a set $\mathcal{P}$ of parametrizations satisfying:
		\begin{enumerate}
			\item All constant maps $\R^p \to X$ belong to $\mathcal{P}$.
			\item If $\{f_i: O_i \to X\}_{i \in I}$ is a family of plots in $\mathcal{P}$ compatible on overlaps, and if they define a map $f: \bigcup_{i} O_i \to X$, then $f \in \mathcal{P}$.
			\item If $f \in \mathcal{P}$ and $g \in C^\infty(O', O)$, then $f \circ g \in \mathcal{P}$.
		\end{enumerate}
	\end{itemize}
	
	A pair $(X, \mathcal{P})$ is called a \textbf{diffeological space}. Given two diffeological spaces $(X, \mathcal{P})$ and $(X', \mathcal{P}')$, a map $f: X \to X'$ is said to be \textbf{smooth} if $f\circ \minor{c} \in \mathcal{P}'$ for every $p \in \mathcal{P}$.
\end{Definition}

The concept of diffeological space was introduced by J.-M. Souriau \cite{Sou}, inspired by early observations of Chen \cite{Chen}. A foundational treatment can be found in \cite{Igdiff}.

\begin{Definition}
	Let $(X,\mathcal{P})$ be a diffeological space. A diffeology $\mathcal{P}'$ on $X$ is called a \textbf{subdiffeology} of $\mathcal{P}$ if $\mathcal{P}' \subset \mathcal{P}$, or equivalently, if the identity map $\mathrm{Id}_X: (X,\mathcal{P}') \to (X,\mathcal{P})$ is smooth.
\end{Definition}

\begin{example}
	Every nonempty set $X$ carries a minimal diffeology, consisting only of constant maps. This is called the \textbf{discrete diffeology}.
\end{example}

The category of diffeological spaces is extremely broad and includes various pathological examples. To control this, one can work within the more structured category of Fr\"olicher spaces.

\begin{Definition}
	A \textbf{Fr\"olicher space} is a triple $(X, \mathcal{F}, \mathcal{C})$ where:
	\begin{itemize}
		\item $\mathcal{C}$ is a set of curves $c: \R \to X$, called \textbf{contours}.
		\item $\mathcal{F}$ is a set of functions $f: X \to \R$ such that $f \circ c \in C^\infty(\R, \R)$ for all $c \in \mathcal{C}$.
		\item Conversely, a curve $c: \R \to X$ belongs to $\mathcal{C}$ if and only if $f \circ c \in C^\infty(\R,\R)$ for all $f \in \mathcal{F}$.
	\end{itemize}
	A map $f: (X, \mathcal{F}, \mathcal{C}) \to (X', \mathcal{F}', \mathcal{C}')$ is \textbf{smooth} if and only if $f \circ c \in \mathcal{C}'$ and $f' \circ f \in \mathcal{F}'$ for all $f' \in \mathcal{F}'$, $c \in \mathcal{C}$.
\end{Definition}

This notion, introduced in \cite{FK}, was later formalized in \cite{KM}. For an early comparison between diffeological and Fr\"olicher settings, see \cite{Ma2006-3}; see also \cite{BKW2024, Ma2013, Ma2018-2, MR2016, Wa}. In particular, \cite{MR2016} explains that:

\begin{quote}
	\emph{Diffeological, Fr\"olicher, and G\^ateaux smoothness coincide on Fr\'echet spaces.}
\end{quote}

To construct a Fr\"olicher structure on a set $X$, one can proceed as follows: given a set of generating functions $\mathcal{F}_g \subset \mathrm{Map}(X, \R)$, define
\[
\mathcal{C} = \{ c: \R \to X \mid f \circ c \in C^\infty(\R, \R)\ \forall f \in \mathcal{F}_g \}, \quad
\mathcal{F} = \{ f: X \to \R \mid f \circ c \in C^\infty(\R, \R)\ \forall c \in \mathcal{C} \}.
\]
We say that $\mathcal{F}_g$ \textbf{generates} the Fr\"olicher structure $(X, \mathcal{F}, \mathcal{C})$.

This structure induces a natural topology on $X$, namely the pullback of the standard topology on $\R$ via the maps in $\mathcal{F}$.

Alternatively, one can begin with a set of generating contours $\mathcal{C}_g$, and define:
\[
\mathcal{F} = \{ f: X \to \R \mid f \circ c \in C^\infty(\R, \R)\ \forall c \in \mathcal{C}_g \}, \quad
\mathcal{C} = \{ c: \R \to X \mid f \circ c \in C^\infty(\R, \R)\ \forall f \in \mathcal{F} \}.
\]

If $X$ is a finite-dimensional smooth manifold, taking $\mathcal{F} = C^\infty(X, \R)$ and $\mathcal{C} = C^\infty(\R, X)$ yields a Fr\"olicher structure whose topology matches the standard manifold topology \cite{KM}.
Moreover, one can define a diffeology on a Fr\"olicher space by:
\begin{equation} \label{neb}
\mathcal{P}_\infty(\mathcal{F}) := \coprod_{p \in \mathbb{N}^*} \left\{ f: D(f) \subset \R^p \to X \ \middle| \ f \text{ is smooth and } \mathcal{F} \circ f \subset C^\infty(D(f), \R) \right\}.
\end{equation}
This is known as the \emph{nebulae diffeology} (cf. \cite{Igdiff}).

\begin{Proposition}[{\cite{Ma2006-3}}] \label{fd}
	Let $(X, \mathcal{F}, \mathcal{C})$ and $(X', \mathcal{F}', \mathcal{C}')$ be Fr\"olicher spaces. A map $f: X \to X'$ is smooth in the Fr\"olicher sense if and only if it is smooth for the associated diffeologies $\mathcal{P}_\infty(\mathcal{F})$ and $\mathcal{P}_\infty(\mathcal{F}')$.
\end{Proposition}

Hence, we have the chain of implications:
\[
\text{smooth manifold} \Rightarrow \text{Fr\"olicher space} \Rightarrow \text{diffeological space}.
\]
These implications can be refined and are further analyzed in \cite{Wa, Ma2006-3}.

\begin{remark}
	The set of contours $\mathcal{C}$ of a Fr\"olicher space $(X, \mathcal{F}, \mathcal{C})$ does not define a diffeology by itself, because a diffeology requires closure under domain restriction. However, one can construct a minimal diffeology $\mathcal{P}_1(\mathcal{F})$ consisting of plots locally of the form $c \circ g$, with $c \in \mathcal{C}$ and $g$ a smooth map. Then, Proposition \ref{fd} remains valid if we replace $\mathcal{P}_\infty$ by $\mathcal{P}_1$. This is based on Boman's theorem \cite[p.~26]{KM}, and detailed in \cite{Ma2013, Ma2006-3, Wa}.
\end{remark}

\begin{Definition}[\cite{Wa}]
	A diffeology $\mathcal{P}_\infty$ defined by a set of functions $\mathcal{F}$ along the lines of {\rm (\ref{neb})} is called \textbf{reflexive}. Such diffeologies are those generated by the smooth functions of a Fr\"olicher structure.
\end{Definition}

We now present additional examples of diffeologies.

\begin{example}
	Let $k \in \mathbb{N}$ and let $(X, \mathcal{F}, \mathcal{C})$ be a Fr\"olicher space. The \textbf{$C^k$-diffeology} $\mathcal{P}_{(k)}$ is defined by:
	\[
	\forall d \in \mathbb{N}^*, \ \forall O \subset \R^d \text{ open,} \quad 
	\mathcal{P}_{(k)}(O) = \left\{ p: O \to X \mid \mathcal{F} \circ p \subset C^k(O, \R) \right\},
	\]
	and we set
	\[
	\mathcal{P}_{(k)} = \bigcup_{d \in \mathbb{N}^*} \bigcup_{\text{open } O \subset \R^d} \mathcal{P}_{(k)}(O).
	\]
	This construction generalizes the finite-dimensional case treated in \cite{Igdiff}.
\end{example}

\begin{Proposition}[\cite{Igdiff, Sou}] \label{prod1}
	Let $(X, \mathcal{P})$ and $(X', \mathcal{P}')$ be diffeological spaces. The \textbf{product diffeology} on $X \times X'$ is defined as the coarsest diffeology for which the projections are smooth. It consists of maps $g: O \to X \times X'$ that decompose as $g = (f, f')$ with $f \in \mathcal{P}$, $f' \in \mathcal{P}'$. This definition extends to arbitrary (even uncountable) products.
\end{Proposition}

\begin{Proposition} \label{prod2}
	Let $(X, \mathcal{F}, \mathcal{C})$ and $(X', \mathcal{F}', \mathcal{C}')$ be Fr\"olicher spaces, with associated diffeologies $\mathcal{P}$ and $\mathcal{P}'$. Then $X \times X'$ admits a natural Fr\"olicher structure whose set of contours is $\mathcal{C} \times \mathcal{C}'$, and whose associated diffeology is the product $\mathcal{P} \times \mathcal{P}'$.
\end{Proposition}

This construction also extends to infinite products.

Let $(X, \mathcal{P})$ be a diffeological space and $X'$ a set. Given a map $f: X \to X'$, the \textbf{pushforward diffeology} $f_*(\mathcal{P})$ is the coarsest diffeology on $X'$ making $f$ smooth. Conversely, if $(X', \mathcal{P}')$ is diffeological and $f: X \to X'$, the \textbf{pullback diffeology} $f^*(\mathcal{P}')$ is the finest diffeology on $X$ for which $f$ is smooth.

\begin{Proposition}[\cite{Igdiff, Sou}] \label{quotient}
	Let $(X, \mathcal{P})$ be a diffeological space and $\sim$ an equivalence relation. Then the quotient $X/\sim$ inherits a diffeology $\mathcal{P}/\sim$ defined as the pushforward of $\mathcal{P}$ via the projection $X \to X/\sim$. Locally, plots are of the form $\pi \circ p$, with $p \in \mathcal{P}$.
\end{Proposition}

Let $X_0 \subset X$ be a subset.

\begin{itemize}
	\item If $X$ is diffeological with $\mathcal{P}$, the \textbf{subset diffeology} on $X_0$ is the restriction $\mathcal{P}_0 = \{ p \in \mathcal{P} \mid \mathrm{Im}(p) \subset X_0 \}$.
	
	\item If $(X, \mathcal{F}, \mathcal{C})$ is Fr\"olicher, we define the induced structure on $X_0$ using the restrictions of maps in $\mathcal{F}$ and contours in $\mathcal{C}$ with image in $X_0$.
\end{itemize}

\begin{example}
	Let $X, Y$ be sets and consider the diffeological space $Y^X$ with the \textbf{functional product diffeology}, defined as the largest diffeology making all evaluation maps $\mathrm{ev}_x: Y^X \to Y$ smooth. Any subspace of $Y^X$ admitting this property inherits a subdiffeology of this functional product diffeology.
\end{example}

\begin{Proposition}[\cite{Igdiff}] \label{cvar}
	Let $X, Y, Z$ be diffeological spaces. Then the canonical identifications
	\[
	C^\infty(X \times Y, Z) \simeq C^\infty(X, C^\infty(Y, Z)) \simeq C^\infty(Y, C^\infty(X, Z))
	\]
	hold, where all function spaces are equipped with the functional diffeology.
\end{Proposition}
\subsection{Groups, vector spaces, algebra and fiber bundles in diffeologies}

Given an algebraic structure, one can naturally define a compatible diffeological or Fr\"olicher structure. For instance, see \cite{Les, Ma2013}, and more specifically \cite[pp.~66--68]{Igdiff} for the case of vector spaces. 

If $\R$ is equipped with its standard diffeology (or Fr\"olicher structure), then an $\R$-vector space $V$ equipped with a diffeology (or Fr\"olicher structure) is called a \textbf{diffeological (resp. Fr\"olicher) vector space} if vector addition and scalar multiplication are smooth maps.

\begin{Definition}
	Let $G$ be a group endowed with a diffeology (resp. a Fr\"olicher structure). Then $G$ is called a \textbf{diffeological (resp. Fr\"olicher) group} if both the multiplication map $G \times G \to G$ and the inversion map $G \to G$ are smooth.
\end{Definition}

This notion applies similarly to rings, modules, algebras, and other algebraic structures where compatibility with the diffeological or Fr\"olicher framework requires the corresponding algebraic operations to be smooth.

Moreover, one can define fiber bundles in the diffeological context. A \textbf{diffeological fiber bundle} is a surjective smooth map $\pi: E \to B$ between diffeological spaces such that locally (in the diffeological sense), $E$ is diffeomorphic to $B \times F$, with $F$ the typical fiber, and the transition functions are smooth in the functional diffeology sense.

Many classical constructions—direct sums, duals, subspaces—have analogs in the diffeological category, but must be handled carefully since the lack of local triviality (in the usual topological sense) may affect their behavior. In particular, when dealing wit tangent and cotangent spaces often belong to extended classes called pseudo-bundles (see e.g. \cite{Ma2025-1}), namely, 
\begin{itemize} 
\item the cotangent space of a diffeological space is defined in a widely accepted way in the literature; but it appears that there is no typical fiber. The \minor{fibers} are vector spaces, but they can be no-isomorphic. For example, the cross has a cotangent bundle with fibers of dimension 1 or 2. This is called a vector pseudo bundle
\item the notion of tangent bundle is even more difficult to study because many non-equivalent versions are actually proposed, and each of them has its own motivations. 
\end{itemize}
For an overview on tangent spaces, cotangent spaces, vector pseudo bundles, we refer to \cite{GMW2023}.

\section{Linearizations from Finite-Dimensional Manifolds to Diffeological Spaces and Infinite-Dimensional Settings}

\subsection{Linearizations of a Manifold}

Bokobza-Haggiag introduces the notion of \emph{linearization} in \cite{BK}, aiming to extend local analytical constructions—such as differential operators—into a global structure on a smooth manifold $M$. Let us first recall the classical concept:

\begin{Definition}
	A \emph{local linearization} at a point $x \in M$ is a diffeomorphism
	\[
	\psi_x : U \subset M \longrightarrow V \subset \mathbb{R}^n,
	\]
	such that $\psi_x(x) = 0$ and $d\psi_x|_x : T_xM \to \mathbb{R}^n$ is the identity (under a fixed identification).
\end{Definition}

The paper \cite{BK} generalizes this to a global notion:

\begin{Definition}[\cite{BK}] \label{lin-I}
	A (global) \emph{linearization} is a smooth map $\nu: X \times X \rightarrow TX$ such that:
	\begin{enumerate}
		\item Let $p_1: X \times X \rightarrow X$ denote the projection onto the first component, and let $\pi: TX \rightarrow X$ denote the tangent bundle projection. Then the following diagram commutes:
		\[
		\begin{tikzcd}
			X \times X \arrow[rr, "\nu"] \arrow[dr, "p_1"'] & & TX \arrow[dl, "\pi"] \\
			& X
		\end{tikzcd}
		\]
		
		\item For all $x \in X$, one has $\nu(x,x) = 0$.
		
		\item For all $x \in X$, the differential $D\nu(x, \cdot)|_{y=x}: T_xX \to T_0(T_xX) \simeq T_xX$ is the identity.
	\end{enumerate}
\end{Definition}

Using classical geometric techniques, one can derive the following properties:

\begin{enumerate}
	\item By the inverse function theorem, a linearization defines a local linearization at each point $x \in X$.
	
	\item There exists a neighborhood $\Omega$ of the diagonal $\Delta(X) \subset X \times X$ such that $\nu|_\Omega$ is a diffeomorphism onto a neighborhood of the zero section of $TX$.
	
	\item The map
	\[
	l_x^y : \xi \in T^*X \mapsto d_y \langle \nu(x, y), \xi \rangle
	\]
	is a smooth linear map from $T^*_y X$ to $T^*_x X$.
	
	\item Every manifold $X$ admits a linearization.
	
	\item A manifold $X$ admits an atlas with affine transition functions if and only if there exists a linearization $\mu$ of $X$ such that, for all $(x,y,z) \in X^3$ with $(x,y), (y,z), (x,z) \in \Omega$, one has
	\[
	l_x^z = l_y^z \circ l_x^y.
	\]
\end{enumerate}

We present an equivalent formulation of Definition~\ref{lin-I} that will be helpful for further generalization.

\begin{Proposition}
	The following condition is equivalent to Definition~\ref{lin-I}.
	
	Define $k : T^*X \times X \to \R$ by
	\[
	k(\xi, y) := \langle \nu(x, y), \xi \rangle, \quad \text{with } x = \pi(\xi).
	\]
	Then $\nu$ is a global linearization if and only if:
	\begin{enumerate}
		\item $k(\xi, x) = 0$ for all $\xi \in T^*X$ (where $x = \pi(\xi)$);
		
		\item The differential $d_y k(\xi, y)|_{y = x} : T^*_x X \to T^*_x X$ is the identity.
	\end{enumerate}
\end{Proposition}

\subsection{Exponential Map and Linearization on a Compact Riemannian Manifold}

Let $(M, g)$ be a compact, connected, orientable, boundaryless Riemannian manifold. By the Hopf–Rinow theorem, $M$ is geodesically complete, so the exponential map
\[
\exp_p : T_p M \to M
\]
is defined on the entire tangent space $T_p M$ for every $p \in M$.

\subsubsection{Properties of the exponential map}

We follow classical references such as \cite{Ber2003, GHL2004}. Let $v \in T_p M$ and consider the geodesic $\gamma(t) = \exp_p(tv)$. Then $d\exp_p|_v$ satisfies:

\begin{itemize}
	\item \textbf{Identity at the origin:}
	\[
	d\exp_p|_{0} = \mathrm{id}_{T_p M}.
	\]
	
	\item \textbf{Critical and regular points:} A vector $v \in T_p M$ is a regular point if $d\exp_p|_v$ is invertible. Otherwise, $v$ is critical, and $\exp_p(v)$ is a conjugate point along the geodesic.
	
	\item \textbf{Local diffeomorphism:} There exists $\varepsilon > 0$ such that $\exp_p$ is a diffeomorphism from $B(0,\varepsilon) \subset T_p M$ onto its image.
	
	\item \textbf{Injectivity radius:} The injectivity radius at $p$ is the supremum of $r > 0$ such that $\exp_p$ is a diffeomorphism on $B(0,r)$. For compact $M$, this radius is uniform and strictly positive.
	
	\item \textbf{Uniform boundedness:} If the Riemann curvature is uniformly bounded, then the norm of $d\exp_p$ can be estimated via comparison theorems (e.g., Rauch’s). There exists a uniform radius $r > 0$ such that $\exp_p$ is a diffeomorphism on $B(0, r) \subset T_p M$.
\end{itemize}

\subsubsection{The Exponential Map Defines a Linearization in the Sense of Bokobza-Haggiag} \label{ss:man}

Let $M$ be as above. Since the injectivity radius $r > 0$ is uniform, the exponential map defines a diffeomorphism:
\[
\exp: B_r = \left\{ v \in TM \ \middle| \ \langle v, v \rangle < r^2 \right\} \to V_r \subset M \times M,
\]
where $V_r$ is a neighborhood of the diagonal $\Delta(M)$ in $M \times M$.

Define $\delta: TM \to V_r$ by
\[
\delta(v) := \left(x, \exp_x \left( \frac{r}{\sqrt{1 + \langle v,v \rangle}} v \right)\right),
\]
where $v \in T_x M$.

Then:
\begin{itemize}
	\item $\delta \in C^\infty(TM, M \times M)$;
	
	\item For any linearization $\nu$ of $M$ in the sense of Bokobza-Haggiag, we have:
	\[
	D_0\left(\left(\frac{1}{r}\nu\right) \circ \delta\right) = \mathrm{Id}_{T_x X} = D_x\left(\delta \circ \left(\frac{1}{r}\nu\right)\right),
	\]
    where $D_0$ is the derivative along the fibers of $T_X$ and $D_x$ is the derivative along the second variable of $X \times X.$
\end{itemize}

Using the Riemannian metric, we identify $TM \cong T^*M$. Hence, both tangent and cotangent bundles may be used interchangeably in the characterization of linearizations.

These properties will serve as a reference point when generalizing linearizations to diffeological spaces.
\section{Extending Classical Optimization Methods in Linearizations}
\subsection{A Generalized Notion of Linearization}

\begin{Definition} \label{d:linearisation}
	Let $(X,\p)$ be a diffeological space. Let $E$ be a diffeological vector pseudo-bundle over $X$ with projection $\pi_E$, not necessarily equipped with local trivializations or smooth partitions of unity, but such that the inclusion map $T^*X \subset E$ is \minor{a morphism} of diffeological pseudo-bundles. A \emph{linearization} of $X$ is a pair $(\nu,\delta)$ of smooth maps {\color{black} \[
  \delta : E \longrightarrow X \times X,
  \qquad
  \nu : X \times X \longrightarrow E,
\]
such that the following properties hold:}
	
	\begin{enumerate}
		\item[(1)] {\color{black} Denoting by $\Delta(X)$ the diagonal of $X\times X,$} $\delta^{-1}_{|T^*X}(\Delta(X))$ is the zero section of $T^*X$, {\color{black} i.e. $\forall \xi \in T^*X \subset E, \delta(\xi) \in \Delta(X) \Leftrightarrow \xi = 0.$}
        \item[(2)] the diagram
		\begin{center}
			\begin{tikzcd}
				E \arrow[rr,"\delta"] \arrow[dr, "\pi_E"'] & & X \times X \arrow[dl, "p_1"] \\
				& X  
			\end{tikzcd}
		\end{center}
		commutes, {\color{black} i.e. the base point is preserved.}
		\item[(3)] $\nu : X \times X \rightarrow E$ is smooth,
		\item[(4)] \label{i:x=y} $\forall (x,y) \in X^2,$ we have $\nu(x,y)=0 \Leftrightarrow x=y$,
		\item[(5)] $\forall x \in X,$ the derivative $d_x\nu(x,.) \neq 0$, {\color{black} where $d_x$ is the de Rham differential, here applied to the map $(\nu(x,.): y \mapsto \nu(x,y)) \in C^\infty(X,E_x). $ Therefore, $d_x\nu(x,.) \in \Omega^1(X,E_x).$}
		\item[(6)] The diagram
		\begin{center}
			\begin{tikzcd}
				X \times X \arrow[rr,"\nu"] \arrow[dr, "p_1"'] & & E \arrow[dl, "\pi_E"] \\
				& X  
			\end{tikzcd}
		\end{center}
		commutes,
		\item[(7)] For every $x \in X$, the induced map from $T^*X$ to $T^*X$ by $(\delta \circ \nu)(x,.) \in C^\infty(X,X)$ restricts to an injective map on $T^*_xX$.
	\end{enumerate}
    {\color{black} We denote such a structure by the pair \((\delta,\nu)\) and say that \( (E,\delta,\nu) \) is a linearization of \(X\).}
\end{Definition}

This definition mirrors the notation in Section \ref{ss:man} to highlight differences with the classical definition by Bokobza-Haggiag. Notably, condition (4) does not generally hold in the classical case when $\nu$ targets $TX$ or $T^*X$. The non-trivial homotopy type of $X$ is a main obstruction such a global property.

This technical obstacle is circumvented by enlarging the target space to a pseudo-bundle $E$ over $X$ containing $T^*X$ as a sub-pseudo-bundle. The main motivation is that $T^*X$ is well-defined in the diffeological setting, unlike the tangent bundle, for which several constructions exist. We now verify the existence of such linearizations in standard cases, {\color{black}namely, when $X$ is a (classical) smooth manifold}:

{\color{black}

\begin{Proposition}[Smooth linearization on a differentiable manifold]
\label{prop:linearization-smooth}
Let $(X,g)$ be a smooth finite-dimensional Riemannian manifold.
Using the musical isomorphism $v \mapsto v^\flat := g_x(v,\cdot)$,
we identify $T_xX$ with $T_x^\ast X$.
Set
\[
E := T^\ast X \oplus \mathbb{R},
\]
viewed as a smooth vector bundle over $X$ via the projection 
$\pi_1 : E \to X$.

Then there exists a smooth map
\[
\nu : X\times X \longrightarrow E
\]
such that:
\begin{enumerate}
  \item[(i)] $\nu(x,x) = (0_x,0)$ for all $x\in X$;

  \item[(ii)] $\nu(x,y) = (0_x,0)$ if and only if $y=x$;

  \item[(iii)] for each $x\in X$, the differential of the first component
  of $\nu$ with respect to the second variable, evaluated at $(x,x)$, is
  the musical isomorphism
  \[
    D_2\bigl(pr_1\circ\nu\bigr)(x,x)
    = (\cdot)^\flat : T_xX \longrightarrow T_x^\ast X ,
  \]
  where $pr_1 : E\to T^\ast X$ denotes the first projection.
\end{enumerate}
In particular, $(E,\nu)$ defines a smooth linearization of $X$ along the
diagonal in the above sense.
\end{Proposition}

\begin{proof}
\emph{Step 1: A Bokobza--Haggiag type linear part in $T^\ast X$.}

Fix a smooth Riemannian metric $g$ on $X$.  
For each $x\in X$, let $\exp_x : \mathcal U_x \subset T_xX \to X$ be the 
exponential map, defined on some neighbourhood $\mathcal U_x$ of $0_x$.
By the tubular neighbourhood theorem for the diagonal
\[
\Delta := \{(x,x) : x\in X\} \subset X\times X,
\]
there exists an open neighbourhood $U$ of $\Delta$ in $X\times X$ and
a smooth map
\[
\xi : U \longrightarrow TX
\]
such that:
\begin{enumerate}
  \item $\xi(x,y)\in T_xX$ for all $(x,y)\in U$;
  \item $\xi(x,x) = 0_x$ for all $x\in X$;
  \item $y = \exp_x\bigl(\xi(x,y)\bigr)$ for all $(x,y)\in U$.
\end{enumerate}
In particular,
\[
D_2\xi(x,x) = \mathrm{Id}_{T_xX}.
\]

Define
\[
\nu_{\mathrm{BH}} : U \longrightarrow T^\ast X,
\qquad
\nu_{\mathrm{BH}}(x,y) := \xi(x,y)^\flat.
\]
Then $\nu_{\mathrm{BH}}$ is smooth on $U$, satisfies
$\nu_{\mathrm{BH}}(x,x)=0_x$, and
\[
D_2\nu_{\mathrm{BH}}(x,x)
= (\cdot)^\flat : T_xX \to T_x^\ast X.
\]

\medskip
\noindent
\emph{Step 2: A smooth function vanishing exactly on the diagonal.}

Define on $U$:
\[
\rho_0(x,y) := \|\xi(x,y)\|_g^2.
\]
Then $\rho_0$ is smooth on $U$, nonnegative, and
\[
\rho_0(x,y) = 0 \Longleftrightarrow (x,y)\in \Delta \cap U.
\]

Choose an open subset $U_0$ such that
\[
\Delta \subset U_0 \subset \overline{U_0} \subset U,
\]
and a smooth cutoff function 
$\chi\in C^\infty(X\times X,[0,1])$ with
\[
\chi \equiv 1 \ \text{on } U_0,
\qquad
\mathrm{supp}(\chi)\subset U.
\]
Define a global smooth map
\[
\rho : X\times X \longrightarrow [0,+\infty)
\]
by
\[
\rho(x,y)
:= \chi(x,y)\,\rho_0(x,y) + \bigl(1-\chi(x,y)\bigr).
\]
Then:
\begin{itemize}
  \item on $U_0$, $\rho=\rho_0$, hence $\rho(x,y)=0$ iff $(x,y)\in\Delta$;
  \item outside $U$, $\chi=0$ and thus $\rho\equiv 1$.
\end{itemize}
Hence $\rho$ is smooth and
\[
\rho(x,y)=0 \Longleftrightarrow (x,y)\in\Delta.
\]

\medskip
\noindent
\emph{Step 3: Global extension of $\nu_{\mathrm{BH}}$.}

Define
\[
\widetilde{\nu}_{\mathrm{BH}}(x,y)
:= \chi(x,y)\,\nu_{\mathrm{BH}}(x,y),
\]
extending $\nu_{\mathrm{BH}}$ smoothly across $U$ and taking it to $0$ outside $U$.
Thus $\widetilde{\nu}_{\mathrm{BH}}$ is smooth on $X\times X$, and coincides
with $\nu_{\mathrm{BH}}$ on $U_0$.

\medskip
\noindent
\emph{Step 4: Definition of the linearization $\nu$.}

Set
\[
\nu(x,y)
:= \bigl(\widetilde{\nu}_{\mathrm{BH}}(x,y),\;\rho(x,y)\bigr)
\in T_x^\ast X \oplus \mathbb{R}.
\]

\medskip
\noindent
\emph{Step 5: Verification of properties (i)--(iii).}

(i)  
For $y=x$, we have $(x,x)\in U_0$.
Hence $\chi(x,x)=1$, $\rho(x,x)=\|\xi(x,x)\|^2=0$, and
$\widetilde{\nu}_{\mathrm{BH}}(x,x)=\nu_{\mathrm{BH}}(x,x)=0_x$.
Thus
\[
\nu(x,x)=(0_x,0).
\]

(ii)  
If $\nu(x,y)=(0_x,0)$, then in particular $\rho(x,y)=0$.
By construction of $\rho$, this implies $(x,y)\in\Delta$, hence $y=x$.
The converse is already proved in (i).
Thus
\[
\nu(x,y)=(0_x,0)\quad\Longleftrightarrow\quad y=x.
\]

(iii)  
Near $(x,x)$, we are in $U_0$, so
\[
pr_1\circ \nu = \widetilde{\nu}_{\mathrm{BH}}
                = \nu_{\mathrm{BH}}.
\]
Therefore
\[
D_2(pr_1\circ \nu)(x,x)
= D_2\nu_{\mathrm{BH}}(x,x)
= (\cdot)^\flat : T_xX \to T_x^\ast X,
\]
which is the required linear isomorphism.

\medskip
This completes the proof.
\end{proof}
\begin{Proposition}[Construction of $\delta$ compatible with $\nu$]
\label{prop:delta-construction}
Let $(X,g)$ be a smooth finite-dimensional Riemannian manifold, and let
\[
E := T^\ast X \oplus \mathbb{R}
\]
as in Proposition~\ref{prop:linearization-smooth}.
Let $\nu : X\times X \to E$ be the smooth map constructed there, i.e.
\[
\nu(x,y) = \bigl(\widetilde{\nu}_{\mathrm{BH}}(x,y),\,\rho(x,y)\bigr),
\]
where $\widetilde{\nu}_{\mathrm{BH}}$ extends the Bokobza--Haggiag-type
linearization and $\rho$ is a smooth function vanishing exactly on the
diagonal $\Delta\subset X\times X$.

Then there exists a smooth map
\[
\delta : E \longrightarrow X\times X
\]
such that:
\begin{enumerate}
  \item[(i)] $\delta(x,0_x,s) = (x,x)$ for all $x\in X$ and $s\in\mathbb{R}$;

  \item[(ii)] there exists an open neighbourhood $U_0$ of the diagonal
  in $X\times X$ such that, for all $(x,y)\in U_0$, writing
  \[
    \nu(x,y) = \bigl(\alpha_{x,y},\, \rho(x,y)\bigr)
    \in T_x^\ast X\oplus\mathbb{R},
  \]
  one has
  \[
    \delta\bigl(x,\alpha_{x,y},\rho(x,y)\bigr) = (x,y);
  \]

  \item[(iii)] for each $x\in X$, the vertical differential of $\delta$
  at $(x,0_x,0)$, restricted to the $T_x^\ast X$-component, induces the
  musical isomorphism
  \[
    D^{\mathrm{vert}}_\alpha\delta(x,0_x,0)
    : T_x^\ast X \xrightarrow{\ \sim\ } \{0\}\oplus T_xX
  \]
  via the identification $T_{(x,x)}(X\times X) \simeq T_xX\oplus T_xX$.
\end{enumerate}
Consequently, $(E,\nu,\delta)$ defines a smooth linearization of $X$ in the
sense that, near the diagonal, $\nu$ and $\delta$ are local inverses in the
second variable and encode the tangent structure via $g$.
\end{Proposition}

\begin{proof}
We use the same geometric data as in
Proposition~\ref{prop:linearization-smooth}. In particular, we have:
\begin{itemize}
  \item an open neighbourhood $U\subset X\times X$ of the diagonal
  $\Delta$;
  \item a smooth map $\xi:U\to TX$ such that
  $\xi(x,y)\in T_xX$, $\xi(x,x)=0_x$, and $y=\exp_x(\xi(x,y))$;
  \item an open subset $U_0$ with 
  $\Delta\subset U_0\subset\overline{U_0}\subset U$;
  \item a cutoff function $\chi\in C^\infty(X\times X,[0,1])$ with
  $\chi\equiv 1$ on $U_0$ and $\mathrm{supp}(\chi)\subset U$;
  \item the function $\rho(x,y)=\chi(x,y)\,\|\xi(x,y)\|_g^2+(1-\chi(x,y))$;
  \item the extended Bokobza--Haggiag map 
  $\widetilde{\nu}_{\mathrm{BH}}(x,y)=\chi(x,y)\,\xi(x,y)^\flat$.
\end{itemize}

\medskip
\noindent\textbf{Step 1: Controlling the exponential domain.}

For each $x\in X$, choose $r(x)>0$ such that the open ball
\[
B_x := \{v\in T_xX : \|v\|_g < r(x)\}
\]
is contained in the domain of $\exp_x$ and $\exp_x:B_x\to \exp_x(B_x)$
is a diffeomorphism onto its image. By standard arguments (using a partition
of unity), we may assume that $x\mapsto r(x)$ is a smooth positive function.

Let $\psi:[0,+\infty)\to[0,1]$ be a smooth function such that
\[
\psi(t) = 1 \quad\text{for } t\le 1,
\qquad
\psi(t) = 0 \quad\text{for } t\ge 2.
\]

\medskip
\noindent\textbf{Step 2: Definition of $\delta$.}

For $(x,\alpha,s)\in E_x = T_x^\ast X\oplus\mathbb{R}$, set
\[
v := \alpha^\sharp \in T_xX,
\qquad
\lambda(x,\alpha) := \psi\!\left(\frac{\|v\|_g}{r(x)}\right),
\qquad
v'(x,\alpha) := \lambda(x,\alpha)\,v.
\]
Then $v'(x,\alpha)$ is smooth in $(x,\alpha)$, and:
\begin{itemize}
  \item if $\|v\|_g \le r(x)$, then $\lambda(x,\alpha)=1$ and $v'=v$;
  \item if $\|v\|_g \ge 2\,r(x)$, then $\lambda(x,\alpha)=0$ and $v'=0$.
\end{itemize}

We now define
\[
\delta : E \longrightarrow X\times X,\qquad
\delta(x,\alpha,s) := \bigl(x,\exp_x(v'(x,\alpha))\bigr).
\]
Since $\exp_x$ is smooth in both $x$ and $v$, and $v'$ is smooth,
$\delta$ is a smooth map.

\medskip
\noindent\textbf{Step 3: Verification of (i).}

If $\alpha=0_x$, then $v=0_x$ and hence $v'(x,0_x)=0_x$ for all $x$.
Therefore
\[
\delta(x,0_x,s) = (x,\exp_x(0_x)) = (x,x),
\]
for all $s\in\mathbb{R}$.
This proves (i).

\medskip
\noindent\textbf{Step 4: Compatibility with $\nu$ near the diagonal.}

Recall that, on $U_0$, we have 
$\chi\equiv 1$, hence
\[
\nu(x,y) = \bigl(\nu_{\mathrm{BH}}(x,y),\rho(x,y)\bigr)
         = \bigl(\xi(x,y)^\flat, \|\xi(x,y)\|_g^2\bigr).
\]
Fix $(x,y)\in U_0$ and set
\[
\alpha_{x,y} := \xi(x,y)^\flat.
\]
Then
\[
v := \alpha_{x,y}^\sharp = \xi(x,y),
\qquad
\|v\|_g = \|\xi(x,y)\|_g.
\]
For $(x,y)$ sufficiently close to the diagonal, we may assume that
$\|v\|_g < r(x)$, so that $\lambda(x,\alpha_{x,y})=1$ and 
$v'(x,\alpha_{x,y}) = v$.
Hence
\[
\delta\bigl(x,\alpha_{x,y},\rho(x,y)\bigr)
= \bigl(x,\exp_x(v'(x,\alpha_{x,y}))\bigr)
= \bigl(x,\exp_x(\xi(x,y))\bigr)
= (x,y).
\]
This establishes (ii) (possibly after shrinking $U_0$ if needed).

\medskip
\noindent\textbf{Step 5: Vertical differential at the zero section.}

Consider the vertical derivative of $\delta$ at $(x,0_x,0)$, restricted to
the $T_x^\ast X$-component. For $\alpha$ small, we have 
$v=\alpha^\sharp$ small in $T_xX$, so that $\|v\|_g < r(x)$ and 
$\lambda(x,\alpha)=1$. Hence $v'(x,\alpha)=v$ in a neighbourhood of
$(x,0_x)$, and
\[
\delta(x,\alpha,0)
= \bigl(x,\exp_x(\alpha^\sharp)\bigr).
\]
Differentiating at $\alpha=0_x$, we obtain
\[
D_\alpha\delta(x,0_x,0)(\beta)
= \bigl(0,\, D\exp_x(0_x)(\beta^\sharp)\bigr)
= \bigl(0,\,\beta^\sharp\bigr),
\]
for all $\beta\in T_x^\ast X$, since $D\exp_x(0_x)=\mathrm{Id}_{T_xX}$.
Under the identification $T_{(x,x)}(X\times X)\simeq T_xX\oplus T_xX$,
this means that $D^{\mathrm{vert}}_\alpha\delta(x,0_x,0)$ maps 
$T_x^\ast X$ isomorphically onto the subspace $\{0\}\oplus T_xX$ via
$\beta\mapsto (0,\beta^\sharp)$, which proves (iii).

\medskip
This shows that $\delta$ has the required properties, and that,
near the diagonal, $\delta$ and $\nu$ are local inverses in the second
variable with the correct linear behaviour encoded by the Riemannian metric.
\end{proof}

}

\subsection{Distance Induced by a Linearization}

We now introduce an intrinsic structure for measuring convergence on $X$ using the linearization $(\delta,\nu)$. Note that the map $\delta$ plays a central role in this construction.

{\color{black}
\begin{Definition}[Distance induced by a linearization]
Let \( X \) be a diffeological space equipped with a linearization \((\delta,\nu)\) and
a smooth function \( \mathrm{gap} : E \to \mathbb{R} \) (called the \emph{gap function})
satisfying:
\begin{itemize}
  \item[(i)] \( \mathrm{gap}(0_x) = 0 \) for all \( x \in X \);
  \item[(ii)] \( \mathrm{gap}(\nu(x,y)) = 0 \) if and only if \(x = y\);
  \item[(iii)] for each fixed \(x\), the map \(y \mapsto \mathrm{gap}(\nu(x,y))\) is smooth.
\end{itemize}
Then the functions
\[
  d_X: (x,y) \in X^2 \mapsto  \sup_{z \in X} \big| \mathrm{gap}(\nu(z,x)) - \mathrm{gap}(\nu(z,y)) \big|, \]
  \[
  \forall x \in x, d_E: (u,v) \in E_x^2 \mapsto   d_X(\delta(u - v)) + d_X(\delta(v - u)).
\]
define pseudo-distances on \(X\) and on each fiber of \(E\), respectively.
When these pseudo-distances are non-degenerate, they are called the
\emph{distances induced by the linearization} \((\delta,\nu)\).
\end{Definition}}
To simplify notation, subscripts are dropped when context permits.

\begin{Theorem}
	Let $X$ be a diffeological space equipped with a linearization $(\delta, \nu)$ and a gap function satisfying Assumption D. Then $d_X$ is a distance on $X$ and $d_E$ is a distance on each fiber of $E$.
\end{Theorem}

\begin{proof}
	We verify the axioms of a metric for $d_X$ and $d_E$:
	
	\begin{enumerate}
		\item \emph{Separating property:}
		\[
		gap(\nu(x,y)) \neq 0 \iff x \neq y.
		\]
		Thus, $d_X(x,x) = 0$, and if $x \neq y$, then choosing $z = x$ gives $d_X(x,y) > 0$.
		
		Similarly, $d_E(u,v) = = 0$ if and only if $u = v$.
		
		\item \emph{Symmetry:} immediate from definitions.
		
		\item \emph{Triangle inequality:} For $x,y,z \in X$ and all $w \in X$:
		\begin{eqnarray*}
		\left| gap(\nu(w,x)) - gap(\nu(w,z)) \right| 
		&\leq &\left| gap(\nu(w,x)) - gap(\nu(w,y)) \right| \\
        &&+ \left| gap(\nu(w,y)) - gap(\nu(w,z)) \right|.
		\end{eqnarray*}
		Taking suprema yields $d_X(x,z) \leq d_X(x,y) + d_X(y,z)$.
		\end{enumerate}
	\end{proof}

{\color{blue}
\begin{remark}
Convexity of the gap function is not required for these properties.
\end{remark}}

\begin{Definition}
Let $\tilde{X}$ denote the metric completion of $X$ with respect to $d_X$.
\end{Definition}

\section{Gradient and steepest descent methods... (Almost) Without gradient}

Let us now consider a minimization problem of a given functional $F: X \rightarrow \R$ or an equation $F(u) = 0$ on a diffeological space\footnote{We recall that solving $F(x)=0$ is equivalent to solving the minimization problem for $|F|.$}. We define the following basic notion:

\begin{Definition} \label{d:method}
	A \textbf{method} $\mathcal{M}$ is a mapping that depends on:
	\begin{itemize}
		\item a smooth function $f \in C^\infty(I; \mathbb{R})$, where $I \subset \mathbb{R}$ is compact,
		\item and a point $x \in I$,
	\end{itemize}
	and produces a point $\mathcal{M}_f(x) \in I$ such that:
	\begin{enumerate}
		\item $\forall x \in I,\quad f(\mathcal{M}_f(x)) \leq f(x)$,
		\item $\forall x \in I,\quad f(\mathcal{M}_f(x)) = f(x) \iff \mathcal{M}_f(x) = x$,
		\item \label{iii-method} $\mathcal{M}_f$ is smooth on any subinterval $J \subset I$ where $df \neq 0$.
	\end{enumerate}
\end{Definition}

By this definition, we intend to define a method of the type of the gradient methods, Newton methods for convex functions, and more generally \minor{any descent method with exact line search} that (often approximatively at a finite number of iterations) intend to produce the minimum of a smooth function $f$ of prescribed type on a compact interval.  
\subsection{The Algorithm Associated to a Method $\mathcal{M}$}

To provide context, we first sketch the general principle of the method. The core idea is to build a smooth curve $\gamma$ on $X$, replacing the gradient direction used in classical optimization by a path derived from the linearization.

This path is defined by the linearization map $\delta$, such that:
\[
\gamma(t) = \delta(t \, d_xF).
\]
{\color{black}
\vspace{12pt}
\begin{algorithm}[H]
\caption{Diffeological line-search method induced by a linearization}\label{alg:main}
\KwData{Initial point $x_0 \in X$, maximal number of iterations $N_{\max}$, 
compact search interval $I = [0,T]\subset\mathbb{R}$.}
\KwResult{A sequence $(x_n)_{n\geq 0} \subset X$ with non-increasing values of $F$.}

\For{$n = 0,1,\dots,N_{\max}$}{
  \tcp{Construction of a search curve from the linearization at $x_n$}
  Construct a smooth curve $\gamma_n : I \to X$ with $\gamma_n(0)=x_n$
  using the linearization $(E,\delta,\nu)$ at $x_n$\;
  
  \tcp{Model or evaluation of $F$ along the curve}
  Define the real-valued function
  \[
    \varphi_n(t) := F(\gamma_n(t)), \qquad t\in I.
  \]
  Optionally, introduce a method $\mathcal{M}_F(\gamma_n):I\to\mathbb{R}$ approximating $\varphi_n$,
  and use it only to initialise the line search\;
  
  \tcp{Exact line search on a compact interval}
  Choose $t_n \in I$ such that
  \[
    \varphi_n(t_n) \leq \varphi_n(t) \quad \text{for all } t\in I
    \quad\text{(exact minimizer on $I$).}
  \]
  
  \tcp{Update of the iterate}
  Set
  \[
    x_{n+1} := \gamma_n(t_n).
  \]
  
  \tcp{Monotonicity of the objective}
  Then $F(x_{n+1}) = \varphi_n(t_n) \leq \varphi_n(0) = F(x_n)$\;
}
\end{algorithm}
}

\begin{remark} Our algorithm is an adaptation of the algorithm presented in \cite{GW2020,GW2021a,GW2021b}, where the authors use a retraction map from $TX$ to $X$ in place of $\delta$, and replace the gradient by the differential $df$ in $T^*X$.
\end{remark}

\subsection{Classical A Priori Convergence Results to Weak Solutions}

{\color{black}
\begin{Definition}[Algorithmic stationarity]\label{def:alg-stationary}
Let $X$ be a diffeological space and $F:X\to\mathbb{R}$ a function.
Assume that to each $x\in X$ is associated a smooth search curve
$\gamma_x : I \to X$ on a fixed compact interval $I=[0,T]$, with $\gamma_x(0)=x$.
We say that $x$ is \emph{algorithmically stationary} (with respect to $(\gamma_x)_x$ and $F$)
if for every such curve one has
\[
F(\gamma_x(0)) \leq F(\gamma_x(t)) \quad \text{for all } t\in I.
\]
Equivalently, $t=0$ is a minimizer of $t\mapsto F(\gamma_x(t))$ on $I$.
\end{Definition}
\begin{Theorem}[Limit behaviour of the line-search sequence]\label{thm:main-convergence}
Assume:
\begin{enumerate}
  \item[(H1)] $X$ is continuously embedded into a complete metric vector space $(Y,\|\cdot\|)$,
  and the closure $\overline{X}^Y$ of $X$ in $Y$ is compact.

  \item[(H2)] $F:X\to\mathbb{R}$ admits a continuous extension 
  (still denoted $F$) to $\overline{X}^Y$.

  \item[(H3)] There exists a smooth map 
  \[
    \Gamma : X\times I \longrightarrow X,\qquad (x,t)\mapsto \gamma_x(t),
  \]
  where $I=[0,T]$ is a fixed compact interval, such that for each $x\in X$,
  the curve $t\mapsto\gamma_x(t)$ is the search curve used in Algorithm~\ref{alg:main},
  and $\gamma_x(0)=x$.

  \item[(H4)] Algorithm~\ref{alg:main} performs an exact line search on $I$, i.e.\ for each $n$
  it chooses $t_n\in I$ such that
  \[
    F(\gamma_{x_n}(t_n)) \leq F(\gamma_{x_n}(t)) \quad \text{for all } t\in I,
  \]
  and sets $x_{n+1}:=\gamma_{x_n}(t_n)$.
\end{enumerate}

Let $(x_n)_{n\geq 0}$ be the sequence generated by Algorithm~\ref{alg:main} from an initial point
$x_0\in X$. Then:

\begin{enumerate}
  \item[(i)] The sequence $(F(x_n))_{n\geq 0}$ is non-increasing and converges to
  a limit $L\in(-\infty,F(x_0)]$.

  \item[(ii)] The set $S$ of cluster points of $(x_n)$ in $\overline{X}^Y$ is non-empty and compact.
  Moreover,
  \[
    F(x) = L \quad \text{for every } x\in S.
  \]
  In particular, $F$ is constant on $S$.

  \item[(iii)] Every cluster point $x^\star\in S\cap X$ is algorithmically stationary
  in the sense of Definition~\ref{def:alg-stationary} with respect to $F$ and the
  family of curves $\gamma_x$.
\end{enumerate}
\end{Theorem}
\begin{proof}
We proceed in three steps.

\medskip
\noindent\emph{Step 1: Monotonicity and convergence of $(F(x_n))$.}
By construction of Algorithm~\ref{alg:main} and assumption~(H4), we have
for each $n\geq 0$:
\[
F(x_{n+1}) = F(\gamma_{x_n}(t_n)) 
\leq F(\gamma_{x_n}(0)) 
= F(x_n).
\]
Thus $(F(x_n))_{n\geq 0}$ is a non-increasing sequence in $\mathbb{R}$.
By (H1) and (H2), $F$ is continuous on the compact set $\overline{X}^Y$, hence bounded below.
Therefore $(F(x_n))$ admits a finite limit $L\in(-\infty,F(x_0)]$:
\[
F(x_n) \downarrow L \quad \text{as } n\to\infty.
\]

\medskip
\noindent\emph{Step 2: Cluster points and constancy of $F$ on $S$.}
By (H1), the sequence $(x_n)$ takes values in $X\subset\overline{X}^Y$, whose closure is compact.
Hence $(x_n)$ admits at least one cluster point in $\overline{X}^Y$; let $S$ denote the set
of all such cluster points. Then $S$ is non-empty and compact (as a closed subset of a compact
space).

Let $x^\star\in S$. By definition, there exists a subsequence $(x_{n_k})$ such that
$x_{n_k}\to x^\star$ in $Y$ as $k\to\infty$. By continuity of $F$ on $\overline{X}^Y$ (assumption~(H2)),
we have
\[
F(x_{n_k}) \longrightarrow F(x^\star).
\]
On the other hand, by Step~1,
\[
F(x_{n_k}) \longrightarrow L,
\]
since $(F(x_n))$ converges and $(n_k)$ is an increasing sequence of indices.
Thus $F(x^\star) = L$.

As $x^\star\in S$ was arbitrary, we deduce that
\[
F(x) = L \quad \text{for all } x\in S,
\]
so $F$ is constant on $S$. This proves (ii).

\medskip
\noindent\emph{Step 3: Algorithmic stationarity of cluster points in $X$.}
Let $x^\star\in S\cap X$ be a cluster point which belongs to $X$.
By definition of $S$, there exists a subsequence $(x_{n_k})$ such that
\[
x_{n_k} \longrightarrow x^\star \quad \text{in } Y \text{ as } k\to\infty.
\]

Fix the corresponding search curve $\gamma_{x^\star}:I\to X$ given by (H3), 
with $\gamma_{x^\star}(0)=x^\star$.
Assume, for contradiction, that $x^\star$ is \emph{not} algorithmically stationary.
Then, by Definition~\ref{def:alg-stationary}, there exists $t_0\in I$ such that
\[
F(\gamma_{x^\star}(t_0)) < F(\gamma_{x^\star}(0)) = F(x^\star).
\]
Set
\[
\eta := \frac{1}{2}\Bigl(F(x^\star) - F(\gamma_{x^\star}(t_0))\Bigr) > 0.
\]

By (H3), the map $\Gamma : X\times I\to X$ is smooth, hence continuous in $x$ and $t$.
By (H2), $F$ is continuous in $X$ (and in $\overline{X}^Y$). Therefore the composite
\[
(x,t) \longmapsto F(\Gamma(x,t)) = F(\gamma_x(t))
\]
is continuous on $X\times I$.

By continuity at $(x^\star,t_0)$ and $(x^\star,0)$, there exists $\delta>0$ such that
for all $x\in X$ with $\|x - x^\star\| < \delta$ (in the metric of $Y$) we have
\begin{align*}
\bigl|F(\gamma_x(t_0)) - F(\gamma_{x^\star}(t_0))\bigr| &< \eta,\\
\bigl|F(\gamma_x(0)) - F(\gamma_{x^\star}(0))\bigr| &= \bigl|F(x) - F(x^\star)\bigr| < \eta.
\end{align*}

Since $x_{n_k}\to x^\star$, there exists $K$ such that for all $k\geq K$,
$\|x_{n_k} - x^\star\| < \delta$. For such $k$, we have in particular:
\[
F(\gamma_{x_{n_k}}(t_0)) 
< F(\gamma_{x^\star}(t_0)) + \eta
= F(x^\star) - \eta
< F(x_{n_k}) - \eta + \eta
= F(x_{n_k}),
\]
where we used $F(\gamma_{x^\star}(t_0)) = F(x^\star) - 2\eta$ and 
$|F(x_{n_k}) - F(x^\star)|<\eta$.

Now, by exact line search (assumption~(H4)), we choose $t_{n_k}\in I$ with
\[
F(\gamma_{x_{n_k}}(t_{n_k})) \leq F(\gamma_{x_{n_k}}(t)) \quad \text{for all } t\in I.
\]
In particular,
\[
F(x_{n_k+1}) = F(\gamma_{x_{n_k}}(t_{n_k})) \leq F(\gamma_{x_{n_k}}(t_0)) < F(x_{n_k}).
\]

By Step~1, $(F(x_n))$ converges to $L$. Passing to the limit $k\to\infty$ in the
inequality $F(x_{n_k+1}) < F(x_{n_k})$ and using continuity, we obtain
\[
L \leq \lim_{k\to\infty} F(x_{n_k+1})
\leq \lim_{k\to\infty} F(\gamma_{x_{n_k}}(t_0))
= F(\gamma_{x^\star}(t_0))
< F(x^\star) = L,
\]
which is a contradiction.

Hence our assumption was false, and $x^\star$ must satisfy
\[
F(\gamma_{x^\star}(0)) \leq F(\gamma_{x^\star}(t)) \quad \text{for all } t\in I,
\]
i.e.\ $x^\star$ is algorithmically stationary in the sense of
Definition~\ref{def:alg-stationary}.  This proves (iii) and completes the proof.
\end{proof}

}

\section{Application to Mapping Spaces with Low Regularity}

The notion of manifolds of smooth maps used here follows the framework in \cite{Ee}; see also \cite{KM}, \cite{Freed1988}. Let $E$ be a smooth vector bundle over a compact, connected, boundaryless manifold $M$ of dimension $m$. Denote by $C^\infty(M,E)$ the space of smooth sections $\sigma$ of $E$.

Let $\K = \R$ or $\mathbb{C}$, and let $N$ be a smooth Riemannian manifold without boundary. Then:

- $C^\infty(M,N)$ denotes the set of smooth maps $f: M \rightarrow N$,

- $C^\infty(M,G)$ denotes the space of smooth maps from $M$ to a Lie group $G$.

The space $C^\infty(M,N)$ is a smooth Fréchet manifold, and $C^\infty(M,G)$ is a smooth Fréchet Lie group.

The connected components of $C^\infty(M,N)$ correspond to the homotopy classes of based maps $f: M \rightarrow N$. Note that two such components need not be modeled on the same topological vector space. However, if $f$ and $g$ lie in the same path-connected component, then the bundles $f^*TN$ and $g^*TN$ are isomorphic. The atlas structure of $C^\infty(M,N)$ can be found e.g. in \cite{Ma2006}.

\subsection{Spaces of Mappings with Low Regularity}

In the following, we consider a complex algebra $\mathcal{A}$ equipped with the Hermitian matrix product $(A,B) \mapsto \operatorname{tr}(AB^*)$. When unambiguous, we denote the corresponding norm by $\| \cdot \|$ or $\| \cdot \|_{\mathcal{A}}$. {\color{black} Defining Sobolev classes of functions between two manifolds, and studying the properties of the space of functions, is a long-standing problem, coming from the investigations on the structures of smooth infinite dimensional manifolds \cite{Ee} and rapidly raised in various contexts, see e.g. \cite{BZ1988,Bet1991,BM2021,Det2025,EeL1978,Mir2017}. Depending on the definitions, the space of smooth mappings is dense, or not, in these function spaces with low regularity. For our purpose,   
we follow and generalize \cite{Ma2019,GMW2023} and we choose to follow a definition based on density arguments:}

\begin{Definition} \label{d:sobolev} Let $(s,p) \in \mathbb{R} \times (1,+\infty)$, let $M$ be a compact Riemannian manifold without boundary, and let $N$ be a Riemannian submanifold embedded in $\mathcal{M}$. {\color{black} We define $W^{s,p}(M,\mathcal{A})$ as the completion of $C^\infty(M,\mathcal{A})$ for the norm \[f \mapsto \left(\int_M || (1 + \Delta)^{s/2} f ||_\mathcal{A}^p\right)^{1/p}\] where $\Delta$ is the positive Laplacian.}
	Then, the space $W^{s,p}(M,N)$ is defined as the completion of $C^\infty(M,N)$ in $W^{s,p}(M,\mathcal{A})$ {\color{black} for the strong topology.}
\end{Definition}

We emphasize that the classical condition $s < \dim(M)/p$ \textcolor{black}{(necessary for the existence of an atlas \cite{Ee})} is not assumed here. When $s \leq \dim(M)/p$, this results in a weaker differential structure on $W^{s,p}(M;\mathcal{A})$, as detailed in \cite{Ma2019}.

\begin{Proposition}
	The space $W^{s,p}(M,N)$ is a diffeological space equipped with the subset diffeology inherited from $W^{s,p}(M,\mathcal{A})$.
\end{Proposition}

This completion coincides with the $\overline{X}$ construction used in earlier sections.

{\color{black}
\begin{remark}[Completions of \( C^\infty(M,N) \) in Sobolev-type spaces] The choices made in Definition \ref{d:sobolev} may raise many questions and objections due to various considerations on completion, topology and even intuition on function spaces of low regularity, as already mentioned. In particular, a superficial reading may consider this definition as misleading or false considering \cite{BZ1988,Bet1991}. For these reasons, we feel the need to precise the few following aspects.
Let \( M \) and \( N \) be compact smooth manifolds and \( 1 \le p < \infty \).
\begin{itemize}
  \item[(a)] For \( s \in \mathbb{R}_+^* \), the Sobolev-type space \( W^{s,p}(M,N) \) is classically defined as follows for $p$ large enough:
 \( W^{s,p}(M,N) \) is the completion of \( C^\infty(M,N) \)
  with respect to the metric induced by a fixed embedding \( N \hookrightarrow \mathbb{R}^k \)
  and the standard Sobolev norm on \( W^{s,p}(M,\mathbb{R}^k) \), which perfectly fits with pour definition.
  \item[(b)] If \( s = 0 \), then \( W^{0,p}(M,N) = L^p(M,N) \), 
  \item[(c)] If \( s < 0 \), one can make an objection, because \( W^{s,p}(M, \mathcal{A}) \) is defined as the dual space
  \( (W^{-s,p'}(M,\mathcal{A}))' \) and this suggests that the ``good'' topology for completionis the weak one  whenever this makes sense, in particular, when \(N\)
  is a compact subset of a Banach space, this dual coincides with the completion
  of \( C^\infty(M,N) \) in the weak topology.
\end{itemize}
However, these are features that appear with choices in the definitions of \( W^{s,p}(M, \mathcal{A}) \) different from Definition \ref{d:sobolev}. However, we can anyway conclude that 
the resulting space \( W^{s,p}(M,N), \) following any of the definitions existing in the literature, inherits, as a subspace of a locally convex topological vector space, a natural diffeological structure
as the subset diffeology induced from \( W^{s,p}(M,\mathcal{A}) \), where
\( \mathcal{A} \) is a linear ambient space containing \(N\).
\end{remark}}

\subsection{From Linearizations on $N$ to Linearizations on $C^\infty(M,N)$ and the Equation $F(u) = 0$ in $W^{s,p}(M,N)$}

Since $N$ is compact, it admits a linearization $(\delta,\nu)$ between $TN \times \mathbb{R}$ and $N \times N$.

\begin{Definition}
	Let $f \in C^\infty(M,N)$. Then:
	\begin{itemize}
		\item For all $g \in C^\infty(M,N)$, define $\nu(f,g) : x \mapsto \nu(f(x),g(x)) \in f^*(T^*N) \times \mathbb{R}$.
		\item For any section $s$ of $f^*(T^*N) \times \mathbb{R}$, define $\delta(s) = (f,g)$ such that for all $x \in M$, $(f(x),g(x)) = \delta(s(x))$.
	\end{itemize}
\end{Definition}

\begin{Proposition}
	The maps
	\[
	\delta \in C^\infty(C^\infty(M,TN \times \mathbb{R}), C^\infty(M,N) \times C^\infty(M,N)),
	\quad
	\nu \in C^\infty(C^\infty(M,N) \times C^\infty(M,N), C^\infty(M,TN \times \mathbb{R}))
	\]
	define a linearization of $C^\infty(M,N)$.
\end{Proposition}

We now consider a functional equation:
\[
F(u) = 0, \quad \text{with } F : D \subset W^{s,p}(M,N) \rightarrow \mathbb{R},
\]
where $D \cap C^\infty(M,N)$ is dense in $D$, and $D$ is closed in $W^{s,p}(M,N)$. We assume that:

\vskip 12pt
\hrule
\vskip 6pt
\noindent \textbf{Assumption R.}
The linearization $(\delta,\nu)$ on $C^\infty(M,N)$ restricts to a linearization on the subset $D \cap C^\infty(M,N)$ equipped with the subset diffeology.
\vskip 6pt
\hrule
\vskip 12pt

This assumption is non-trivial in general. For instance, if $D \cap C^\infty(M,N)$ inherits the discrete diffeology, then $T^*C^\infty(M,N) = C^\infty(M,N)$ and the cotangent space has trivial fibers. This highlights the potential limitations of such restrictions, and suggests that constructing a dedicated linearization directly on $D \cap C^\infty(M,N)$ may be more effective for algorithmic implementations.

In what follows, we restrict ourselves to the case $D = W^{s,p}(M,N)$ and $D \cap C^\infty(M,N) = C^\infty(M,N)$.

{\color{black}
\begin{Theorem}[Structural properties of $W^{s,p}(M,N)$]\label{thm:Ws-p-structure}
Let $M$ be a compact smooth manifold of dimension $m$, and let 
$N\subset A$ be a compact embedded submanifold of a finite-dimensional
vector space $A \simeq \mathbb{R}^k$.
Fix $1\leq p<\infty$ and $s\in\mathbb{R}$.

\begin{enumerate}
  \item[(i)] For $s\geq 0$, we define $W^{s,p}(M,A)$ as the usual Sobolev--Slobodeckij
  (or Bessel potential) space ([Adams--Fournier, Sobolev Spaces, Chap.~7]). 
  For $s<0$, we set
  \[
    W^{s,p}(M,A) := \bigl(W^{-s,p'}_0(M,A)\bigr)^\ast,
  \]
  where $p'$ is the Hölder conjugate of $p$ and $W^{-s,p'}_0(M,A)$ is the closure
  of $C^\infty_0(M,A)$ in $W^{-s,p'}(M,A)$
  (cf.\ e.g.\ Adams--Fournier, Thm.~7.48).

  \item[(ii)] If $s<0$, the inclusion 
  \[
    C^0(M,A) \hookrightarrow W^{s,p}(M,A)
  \]
  is continuous. More precisely, since $M$ is compact, every continuous 
  $A$-valued map is in $L^p(M,A)$, and the canonical injection
  $L^p(M,A) \hookrightarrow W^{s,p}(M,A)$ (viewing $L^p$ as a subspace of
  $W^{s,p}$ via distributions) is continuous; hence the composition
  $C^0(M,A) \hookrightarrow L^p(M,A) \hookrightarrow W^{s,p}(M,A)$ is
  continuous.

  \item[(iii)] The space $C^\infty(M,N)\subset C^0(M,N)$ is bounded in
  $C^0(M,N)$: indeed, since $N$ is compact, there exists $R>0$ such that
  $\|u(x)\|_A \leq R$ for all $u\in C^\infty(M,N)$ and all $x\in M$, so
  \[
    \|u\|_{C^0(M,A)} \leq R
    \quad \text{for all } u\in C^\infty(M,N).
  \]
  In particular, for $s<0$, the inclusion
  $C^\infty(M,N)\hookrightarrow W^{s,p}(M,A)$ is bounded by (ii).

  \item[(iv)] Define $W^{s,p}(M,N)$ as the closure of $C^\infty(M,N)$ in
  $W^{s,p}(M,A)$:
  \[
    W^{s,p}(M,N) := \overline{C^\infty(M,N)}^{\,W^{s,p}(M,A)}.
  \]
  Then $W^{s,p}(M,N)$ is a closed subset of $W^{s,p}(M,A)$ and the canonical inclusion
  \[
    W^{s,p}(M,N) \hookrightarrow W^{s,p}(M,A)
  \]
  is continuous and maps bounded subsets of $W^{s,p}(M,N)$ into bounded
  subsets of $W^{s,p}(M,A)$.
\end{enumerate}

No general relative compactness statement for $W^{s,p}(M,N)$ in
$W^{s,p}(M,A)$ is claimed here; in our applications 
we instead assume directly that the relevant sublevel sets of the functional
are relatively compact in $W^{s,p}(M,A)$.
\end{Theorem}

\begin{proof}
Item (i) is the standard definition of Sobolev spaces of positive and negative
order; see e.g.\ Adams--Fournier, Sobolev Spaces, Chap.~7, in particular Thm.~7.48
for the duality characterization when $s<0$.

For (ii), since $M$ is compact, the inclusion $C^0(M,A)\hookrightarrow L^p(M,A)$
is continuous. The identification of $L^p(M,A)$ as a subspace of $W^{s,p}(M,A)$
for $s<0$ is again standard in the theory of negative-order Sobolev spaces,
and the corresponding inclusion is continuous; composing both inclusions yields
the claim.

For (iii), the image of $M$ under any $u\in C^\infty(M,N)$ is contained in the
compact set $N$, hence the $C^0$-norm of $u$ is uniformly bounded in terms of
$\sup_{y\in N}\|y\|_A$. Together with (ii), this yields boundedness in
$W^{s,p}(M,A)$ when $s<0$.

Item (iv) follows from the definition of $W^{s,p}(M,N)$ as a closure in the
Banach space $W^{s,p}(M,A)$: closedness and continuity of the inclusion are
immediate. The last sentence is included to emphasize that, for the purposes
of our convergence results, we do not rely on any general compactness property
of $W^{s,p}(M,N)$ in $W^{s,p}(M,A)$, but instead assume compactness at the level
of sublevel sets.
\end{proof} }
\section{Outlook: fixed-point perspectives}\label{sec:outlook}

{\color{black} The present work develops a diffeological framework for line-search
methods based solely on evaluations of the objective function along
search curves induced by a linearization.
The resulting convergence theory establishes monotonicity of the
objective values and algorithmic stationarity of cluster points under
very mild assumptions.

A natural direction for future research is to reinterpret the update
rule $x_{n+1}=\gamma_{x_n}(t_n)$ as a fixed-point iteration
$x_{n+1}=T(x_n)$ generated by the linearization; here $T$ depends
nonlinearly on both the choice of the search curve and on the outcome
of the line search.
Such interpretations arise in classical optimization when the
search direction is derived from a gradient or Newton-type model.

In the present diffeological context, establishing a genuine fixed-point
theory for the map $T$ would require additional hypotheses ensuring,
for instance, compactness or contractivity in a complete metric model.
These assumptions are well outside the scope of the present article
and are not needed for our convergence results.
Developing such a theory could nevertheless be useful in applications
to nonlinear PDEs or variational problems where a diffeological
linearization carries geometric information; we leave this perspective
for future work.}

\vskip 12pt

\paragraph{\bf Acknowledgements:} J.-P.M was partially
supported by ANID (Chile) via the FONDECYT grant \#1201894 during the first stages of this work.
He also thanks the France 2030 framework programme Centre Henri Lebesgue ANR-11-LABX-0020-01 
for creating an attractive mathematical environment.

\vskip 12pt

\paragraph{\bf Declaration of generative AI and AI-assisted technologies in the writing process}

During the preparation of this work the author used ChatGPT in order to smoothen the expression in English. After using this tool/service, the author reviewed and edited the content as needed and takes full responsibility for the content of the publication.

\vskip 12pt

\paragraph{\bf Data availability statement} No dataset is related to this work. 

\end{document}